\documentclass[fullpage, 11pt]{article}

\usepackage[dvips]{graphicx}

\usepackage[english]{babel}
\usepackage{amsfonts}
\usepackage{amssymb}
\usepackage{amsmath}
\usepackage{latexsym}
\usepackage{amsthm}
\usepackage{epsfig}
\usepackage{graphicx}
\usepackage{color}
\usepackage{subfigure}

\newcommand{\bb}{\mathbb}
\newcommand{\conv}{\mathrm{conv}}

\newcommand{\lin}{\mathrm{lin}}

\newcommand{\R}{\bb R}

\newcommand{\Z}{\bb Z}

\newcommand{\rec}{\mathrm{rec}}

\newcommand{\old}[1]{{}}

\newtheorem{prop}{Proposition}
\newtheorem{theorem}[prop]{Theorem}

\newtheorem{lemma}[prop]{Lemma}

\newtheorem{claim}{Claim}

\newtheorem{remark}[prop]{Remark}

\newcommand{\sm}{\setminus}

\def\st{\,|\,}

\old{ \addtolength{\oddsidemargin}{-30pt}
\addtolength{\evensidemargin}{-30pt} \addtolength{\textwidth}{60pt}

\addtolength{\topmargin}{-22pt} \addtolength{\textheight}{32pt} }

\begin{document}
\title{Unique lifting of integer variables in minimal inequalities}

\author{%
Amitabh Basu${}^{1}$, \; Manoel Camp\^elo${}^{2,6}$, \; Michele
Conforti${}^{3,8}$, \\ G\'erard Cornu\'ejols${}^{4,7}$,\; Giacomo
Zambelli${}^{5,8}$}

\maketitle

\begin{abstract}
This paper contributes to the theory of cutting planes for mixed integer linear programs (MILPs). Minimal valid inequalities are well understood for a relaxation of an MILP in tableau form where all the nonbasic variables are continuous; they are derived using the gauge function of maximal lattice-free convex sets. In this paper we study lifting functions for the nonbasic {\em integer} variables starting from such minimal valid inequalities. We characterize precisely when the lifted coefficient is equal to the coefficient of the corresponding continuous variable in every minimal lifting (This result first appeared in the proceedings of IPCO 2010). The answer is a nonconvex region that can be obtained as a finite union of convex polyhedra. We then establish a necessary and sufficient condition for the uniqueness of the lifting function.
\end{abstract}



\footnotetext[1] {Department of Mathematics, University of California, Davis, CA 95616. abasu@math.ucdavis.edu}

\footnotetext[2] {Departamento de Estat\'\i stica e Matem\'atica Aplicada,
Universidade Federal do Cear\'a, Brazil. \\ mcampelo@lia.ufc.br}

\footnotetext[3] {Dipartimento di Matematica Pura e Applicata, Universit\'a di Padova, Via Trieste 63, 35121 Padova, Italy. \\ conforti@math.unipd.it}

\footnotetext[4] {Tepper School of Business, Carnegie Mellon University, Pittsburgh, PA 15213. gc0v@andrew.cmu.edu}

\footnotetext[5] {London School of Economics and Political Sciences, Houghton Street, London WC2A 2AE. G.Zambelli@lse.ac.uk}

\footnotetext[6] {Supported by CNPq Brazil.}

\footnotetext[7] {Supported by  NSF grant CMMI1024554 and ONR grant
N00014-09-1-0033.}

\footnotetext[8] {Supported by  the Progetto di Eccellenza 2008-2009
of the Fondazione Cassa Risparmio di Padova e Rovigo.}

\section{Introduction}\label{sec:gen-lift}

In the context of mixed integer linear programming, there has been a renewed interest recently in the study of cutting planes  that cut off a basic solution of the linear programming relaxation. More precisely, consider a mixed integer linear set in which the variables are partitioned into a basic set $B$ and a nonbasic set $N$, and $K \subseteq B \cup N$ indexes the integer variables:

\begin{equation} \label{basic}
\begin{array}{rrcll}
& x_i & = &  f_i - \sum_{j \in N}  a_{ij} x_j & \mbox{ for } i \in B \\
& x & \geq & 0 \\
& x_k & \in & \mathbb{Z} & \mbox{ for } k \in K.
\end{array}
\end{equation}

 Let $X$ be the relaxation of \eqref{basic} obtained by dropping the nonnegativity restriction on all the basic variables $x_i$, $i \in B$. The convex hull of $X$   is the
{\em corner polyhedron} introduced by Gomory \cite{gom} (see also \cite{gomjohn}). Note that, for any $i \in B \setminus K$, the equation $x_i  =   f_i - \sum_{j \in N}  a_{ij} x_j$ can be removed from the formulation of $X$ since it just defines variable $x_i$. Therefore, throughout the paper, we will assume $B \subseteq K$, i.e. all basic variables are integer. Andersen, Louveaux, Weismantel and Wolsey \cite{alww} studied the corner polyhedron when $|B|=2$ and $B = K$, i.e. all nonbasic variables are continuous. They give a complete characterization of the corner polyhedron using intersection cuts (Balas \cite{balas}) arising from splits, triangles and quadrilaterals. This very elegant result has been extended to $|B|>2$  and $B = K$ by showing a correspondence between minimal valid inequalities and maximal lattice-free convex sets \cite{bccz}, \cite{BorCor}. These results and their extensions \cite{bccz2}, \cite{DW} are best described in an infinite model, which we motivate next.

\subsection{The Infinite Model}

A classical family of cutting planes for \eqref{basic} is that of Gomory mixed integer cuts. For a given row $i \in B$ of the tableau, the Gomory mixed integer cut is of the form $\sum_{j \in N \setminus K} \psi(a_{ij}) x_j + \sum_{j \in N \cap K} \pi(a_{ij}) x_j \geq 1$ where $\psi$ and $\pi$ are functions given by simple formulas. A nice feature of the Gomory mixed integer cut is that, for fixed $f_i$, the same functions $\psi$, $\pi$ are used for any possible choice of the $a_{ij}$s in \eqref{basic}. It is well known that the Gomory mixed integer cuts are also valid for $X$. More generally, let $a^j$ be the vector with entries $a_{ij}$, $i \in B$; we are interested in pairs $(\psi, \pi)$ of functions such that the inequality $\sum_{j \in N \setminus K} \psi(a^{j}) x_j + \sum_{j \in N \cap K} \pi(a^{j}) x_j \geq 1$ is valid for $X$ for any possible number of nonbasic variables and any choice of the nonbasic coefficients $a_{ij}$. Since we are interested in nonredundant inequalities, we can assume that the function $(\psi, \pi)$ is pointwise minimal. While a general characterization of minimal valid functions seems hopeless (see for example \cite{bccz08}), when $N \cap K = \emptyset$ the minimal valid functions $\psi$ are well understood in terms of maximal lattice-free convex sets, as already mentioned. Starting from such a minimal valid function $\psi$, an interesting question is how to generate a function $\pi$ such that $(\psi, \pi)$ is valid and minimal. Recent papers \cite{ccz}, \cite{dw} study when such a function $\pi$ is unique. Here we prove two theorems that generalize and unify results from these two papers.

In order to formalize the concept of valid function $(\psi,\pi)$, we introduce the following infinite model. In the setting below, we also allow further linear constraints on the basic variables. Let $S$ be the set of integral points in some rational polyhedron in
$\R^n$ such that $\dim(S)=n$ (for example, $S$ could be the set of nonnegative integer points). Let $f \in \R^n \setminus S$. Consider
the following infinite relaxation of \eqref{basic}, introduced in~\cite{DW}.

\begin{equation} \label{mod:int}
x=\begin{array}[t]{l}
\displaystyle{f + \sum_{r\in \R^n}r s_r + \sum_{r\in \R^n}r y_r,} \\
x\in S,\\
s_r\in \R_+,\, \forall r\in \R^n, \\
y_r\in \Z_+, \, \forall r\in \R^n,\\
s,y \text{ have finite support}
\end{array}
\end{equation}
where the nonbasic continuous variables have been renamed $s$ and the nonbasic integer variables have been renamed $y$, and where an infinite dimensional vector has {\em finite support} if it has a finite number of nonzero entries. Given two functions $\psi, \pi : \R^n \rightarrow \R$, $(\psi,\pi)$ is said
to be {\em valid} for (\ref{mod:int}) if the inequality
$\sum_{r\in\R^n}\psi(r)s_r + \sum_{r\in\R^n}\pi(r)y_r\geq 1$ holds
for every $(x,s,y)$ satisfying (\ref{mod:int}). We also consider the infinite model where we only have continuous nonbasic variables.

\begin{equation} \label{mod:cont}
x=\begin{array}[t]{l}
\displaystyle{f + \sum_{r\in \R^n}r s_r} \\
x\in S,\\
s_r\in \R_+,\: \forall r\in \R^n,\\
s \text{ has finite support.}
\end{array}
\end{equation}

A function $\psi : \R^n \rightarrow \R$ is said to be {\em valid}
for \eqref{mod:cont} if the inequality $\sum_{r\in\R^n}\psi(r)s_r
\geq 1$ holds for every $(x,s)$ satisfying \eqref{mod:cont}. Given a valid
function $\psi$ for \eqref{mod:cont}, a function $\pi$ is a {\em
lifting} of $\psi$ if $(\psi, \pi)$ is valid for (\ref{mod:int}).
One is interested only in (pointwise) {\em minimal valid functions}, since non-minimal ones are implied by some minimal valid function.
If $\psi$ is a minimal valid function for
\eqref{mod:cont} and $\pi$ is a lifting of $\psi$ such that $(\psi,
\pi)$ is a minimal valid function for (\ref{mod:int}) then we say that $\pi$ is
a {\em minimal lifting} of $\psi$. It can be shown, using Zorn's Lemma, that for every lifting $\pi$ of $\psi$ there exists some minimal lifting $\pi'$ of $\psi$ such that $\pi'\leq \pi$.

\subsection{Sequence Independent Lifting and Unique Lifting Functions}

While minimal valid functions for \eqref{mod:cont} have a simple characterization~\cite{bccz2}, minimal valid functions for (\ref{mod:int}) are not well understood. A general idea to derive minimal valid functions for~\eqref{mod:int} is to start from some minimal valid function $\psi$ for~\eqref{mod:cont}, and construct a minimal lifting $\pi$ of $\psi$. While there is no general technique to compute such a minimal lifting $\pi$, it is known that there exists a region $R_\psi$, containing the origin in its interior, where $\psi$ coincides with $\pi$ for any minimal lifting $\pi$. This latter fact was proved by Dey and Wolsey~\cite{dw} for the case of $S=\Z^2$, and by Conforti, Cornu\'ejols and Zambelli~\cite{ccz} for the general case. In the latter paper, the authors describe the set $R_\psi$ in the case where $\psi$ is defined by a simplicial maximal lattice-free polytope. In this paper we give a precise description of the region $R_\psi$ (Theorem~\ref{thm:main}) in general (This result first appeared in the proceedings of IPCO 2010 \cite{ipco}). The importance of the region $R_\psi$ comes from the fact that for any ray $r \in R_\psi$, the minimal lifting coefficient $\pi(r)$ is unique, i.e. it is the same for every minimal lifting. Thus, we get {\em sequence independent} lifting coefficients for the rays $r$ in $R_\psi$. Moreover, these coefficients can be computed directly from the function $\psi$ for which we have more direct tools. These ideas are related to the results of Balas and Jeroslow~\cite{baljer}.

Furthermore, it is remarked in~\cite{ccz} that, if $\pi$ is a minimal lifting of $\psi$, then $\pi(r)=\pi(r')$ for every $r,r'\in\R^n$ such that $r-r'\in\Z^n\cap\lin(\conv(S))$. Therefore the coefficients of any minimal lifting $\pi$ are uniquely determined in the region $R_\psi + (\Z^n\cap\lin(\conv(S)))$ (throughout this paper, we use  $+$ to denote the Minkowski sum of two sets). In particular, whenever $\R^n$ can be covered by translating $R_\psi$ by integer vectors in $\lin(\conv(S))$, the function $\psi$ has a unique minimal lifting $\pi$. As mentioned above, if $\psi$ has a unique minimal lifting, we can compute the best possible coefficients for all the integer variables in our problem, in a {\em sequence independent} manner. Thus, it is very useful to recognize the minimal valid functions $\psi$ with unique minimal liftings. The second main result in this paper (Theorem~\ref{thm:main2}) is to show that, for the case when $S = \Z^n$, the covering property is in fact a necessary and sufficient condition for the uniqueness of minimal liftings : if $R_\psi + \Z^n\neq \R^n$, then there are at least two distinct minimal liftings for $\psi$. Theorem~\ref{thm:main2} thus converts the question of recognizing minimal valid functions $\psi$ that have a unique lifting function to the geometric question of covering $\R^n$ by lattice translates of the region $R_\psi$. This equivalence is utilized by Basu, Cornu\'ejols and K\"oppe to study the unique lifting properties of certain families of minimal valid functions~\cite{bck}.

\section{Overview of the Main Results}

Let $S$ be the set of integral points in some rational polyhedron in
$\R^n$ such that $\dim(S)=n$ (where $\dim(S)$ denotes the dimension of the affine hull of $S$), and let $f \in \R^n \setminus S$.
To state our main results, we need to explain the characterization of minimal valid functions for~\eqref{mod:cont}. Given a polyhedron $P$, we denote by $\rec(P)$ and $\lin(P)$ the recession cone and the lineality space of $P$.

\subsection{Minimal Valid Functions and Maximal Lattice-Free Convex Sets}

We say that a convex set $B\subseteq\R^n$ is {\em $S$-free} if $B$ does not contain any point of $S$ in its interior. When $S=\Z^n$, $S$-free convex sets are called {\em lattice-free convex sets}. A set $B$ is a {\em maximal $S$-free} convex set if it is an $S$-free convex set that is not properly contained in any $S$-free convex set. It was proved in~\cite{bccz2} that maximal $S$-free convex sets are polyhedra containing a point of $S$ in the relative interior of each facet. The following characterization of maximal $S$-free convex sets and the subsequent remark will be needed in the proofs, but the reader can skip them for now.

\begin{theorem}~\cite{bccz2}\label{Th:minimal1} A full-dimensional convex set $B$ is
a maximal $S$-free
convex set if and only if $B$ is a polyhedron that does not
contain any point of $S$ in its interior and each facet of $B$
contains a point of $S$ in its relative interior. Furthermore if
$B\cap \conv(S)$ has nonempty interior, then $\lin(B)$ contains
$\rec(B\cap\conv(S))$ implying that $\rec(B\cap\conv(S))$ = $\lin(B\cap\conv(S))$, and $\lin(B\cap\conv(S))$ is a rational subspace.
\end{theorem}

The above theorem is a generalization of a classical result of Lovasz~\cite{lo}, which considers the case $S=\Z^n$.

\begin{remark} \label{rmk:exisits-delta} The proof of
Theorem~\ref{Th:minimal1} in~\cite{bccz2} implies the following. Given a maximal
$S$-free convex set  $B$, there exists $\delta>0$ such that
no point of $S\setminus B$ is at a distance less than $\delta$ from $B$.
\end{remark}

\old{A convex set with nonempty interior is obviously full-dimensional. }Given an $S$-free polyhedron $B\subseteq \R^n$ containing $f$ in its interior, $B$ can be uniquely written in the form
\begin{equation}\label{eq:unique-rep-B}
B=\{x\in \R^n:\;a_i(x-f)\le 1,\;i\in I\},
 \end{equation}
 \noindent
 where $I$ is a finite set of indices in one-to-one correspondence with the facets of $B$.

Let $\psi_B : \R^n \rightarrow \R$ be the  function defined by
\begin{equation}\label{eq:psiB}
\psi_B(r)= \max_{i\in I} a_ir, \quad \forall r\in \R^n.
\end{equation}

%Note in particular that, since maximal $S$-free convex sets are polyhedra~\cite{bccz2}, the above function is defined for all maximal $S$-free convex sets $B$ containing $f$ in the interior.

\begin{theorem}~\cite{bccz2} \label{Th:minimal2}
If $B$ is a maximal $S$-free convex set containing $f$ in its interior, then $\psi_B$ is a minimal valid
function for~\eqref{mod:cont}.

Conversely,  let $\psi$ be a minimal valid
function for~\eqref{mod:cont}. Then the set \begin{equation*}\label{eq:Bpsi}
B_\psi:=\{x\in\R^n\st \psi(x-f)\leq 1\}\end{equation*} is a maximal $S$-free convex set
containing $f$ in its interior, and $\psi=\psi_{B_\psi}$.\\
\end{theorem}

Theorem~\ref{Th:minimal2} extends results from two seminal papers by Balas~\cite{balas} and Andersen et al.~\cite{alww} to the infinite model~\eqref{mod:cont}. 

It follows easily from the formula  in \eqref{eq:psiB} and Theorem~\ref{Th:minimal2} that minimal valid functions for~\eqref{mod:cont} are {\em sublinear}. In particular, $\psi$ is subadditive, i.e., $\psi(r_1)+\psi(r_2)\ge \psi(r_1+r_2)$ for all $r_1, r_2\in \R^n$ (see also \cite{bccz2}).

\subsection{Minimal Lifting Functions}

This paper has two main contributions, which we state next. Given a
minimal valid function $\psi$ for \eqref{mod:cont},  $B_{\psi}$ defined in Theorem~\ref{Th:minimal2} is a maximal $S$-free convex
set containing $f$ in its interior.
Following \eqref{eq:unique-rep-B}-\eqref{eq:psiB}, it can be uniquely written
as $B_\psi=\{x\in\R^n\st a_i(x-f)\leq 1,\, i\in I\}$ and so $\psi(r)=\max_{i\in I} a_ir$.
Given $x\in \R^n$, let
\begin{equation*}\label{eq:Rx}R(x):=\{r\in\R^n\st
\psi(r)+\psi(x-f-r)=\psi(x-f)\}.\end{equation*}
We define $$R_\psi:=\{r\in\R^n\st \pi(r)=\psi(r) \mbox{ for all minimal liftings } \pi \mbox{ of } \psi\}.$$
\begin{theorem}\label{thm:main}
Let $\psi$ be a minimal valid function  for \eqref{mod:cont}. Then
$
\displaystyle R_\psi=\bigcup_{x\in S\cap B_\psi} R(x).
$
\end{theorem}
Figure~\ref{fig:regions} illustrates the region $R_\psi$ for several examples.
\begin{figure}%[htbp]
\centering \subfigure[A maximal $\Z^2$-free triangle with three integer points] {\label{fig:region1}
\includegraphics[height=2.5in]{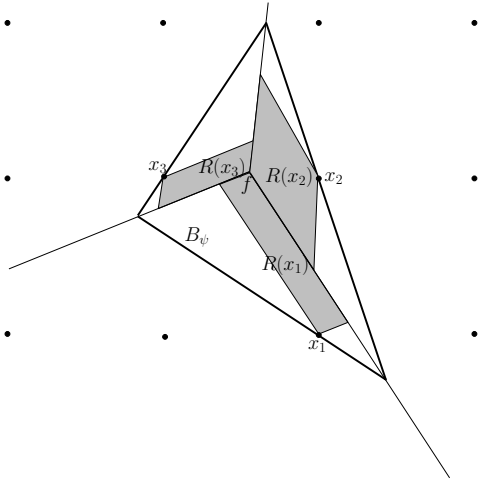}}
\hspace{0.5in}\subfigure[A maximal $\Z^2$-free triangle with integer
vertices] {\label{fig:region4}
\includegraphics[height=2.0in]{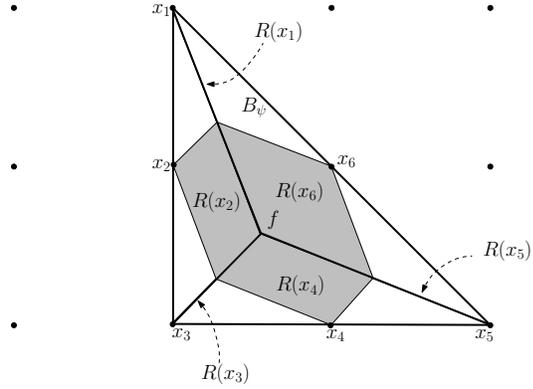}}
\hspace{0.5in}\subfigure[A wedge] {\label{fig:region2}
\includegraphics[height=2.5in]{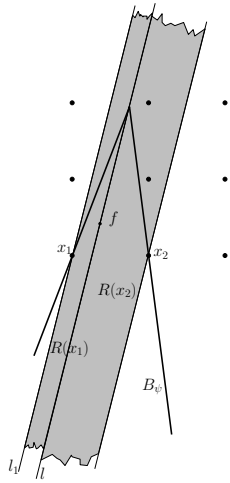}}
\hspace{0.5in} \subfigure[A truncated wedge] {\label{fig:region3}
\includegraphics[height=3.0in]{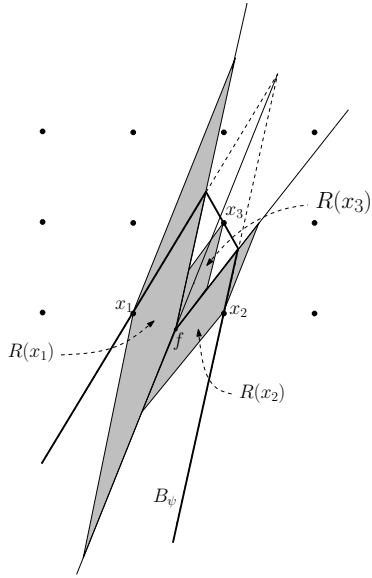}} \caption{Regions $R(x)$ for some
maximal $S$-free convex sets $B$ in the plane and for $x \in S \cap B$. The thick
dark line indicates the boundary of
$B_\psi$. For a particular $x$, the dark gray regions denote $R(x)$ translated by $f$. The jagged
lines in a region indicate that it extends to infinity. For example,
in Figure~\ref{fig:region2}, $R(x_1)$ is the strip
between lines $l_1$ and $l$. Figure~\ref{fig:region4} shows an
example where $R(x)$ is full-dimensional for $x_2, x_4, x_6$, but is
not full-dimensional for $x_1, x_3, x_5$.}\label{fig:regions}
\end{figure}

The second contribution is to give a necessary and sufficient condition for the existence of a {\em unique} minimal lifting function in the case $S=\Z^n$.

\begin{theorem}\label{thm:main2}
Let $S=\Z^n$, and let $\psi$ be a minimal valid function for \eqref{mod:cont}. There exists a unique minimal
lifting $\pi$ of $\psi$ if and only if $R_\psi +
\Z^n=\R^n$.
\end{theorem}

The proofs of Theorems~\ref{thm:main} and~\ref{thm:main2} are given in  Sections~\ref{sec:Thm-lifting-region} and~\ref{sec:unique-lifting}, respectively. We conclude this section with two propositions, used in the proof of  Theorem~\ref{thm:main2}, that are of independent interest. The first gives a geometric description and properties of the regions $R(x)$ and $R_\psi$, while the second (Proposition~\ref{prop:pi_continuous}) states that all minimal liftings are continuous functions.
\begin{prop}\label{prop:region} Let $\psi$ be a minimal valid function
 for \eqref{mod:cont} and let
 $B_\psi =\{x\in \R^n:\;a_i(x-f)\le 1,\;i\in I\}$. Let $L_\psi:=\{r\in \R^n\st a_ir=a_jr \mbox{ for all } i,j\in I \}$. Then
\begin{itemize}
\item[$i)$] For all $x\in\R^n$, $R(x)$ is a polyhedron, namely $R(x)=\{r\in \R^n\st a_ir+a_j(x-f-r)\leq \psi(x-f) \mbox{ for all } i,j\in I\}.$
\item[$ii)$] $\rec(R(x))=\lin(R(x))=L_\psi$ for every $x\in \R^n$.
\item[$iii)$]  $R(x)=R(x')$ for every $x,x'\in \R^n$ such that $x-x'\in L_\psi$.
\item[$iv)$] $R_\psi$ is a finite union of polyhedra.
\item[$v)$] $L_\psi$ is contained in the interior of $R_\psi$.
\end{itemize}
\end{prop}
\begin{proof}
$i)$ Given $x\in \R^n$, let $h\in I$ such that $\psi(x-f)=a_h(x-f)$. Then, for all $r\in\R^n$,
$$\psi(x-f)=a_h (x-f)=a_h r+a_h (x-f-r)\leq \max_{i\in I}a_i r+\max_{i\in I}a_i(x-f-r)=\psi(r)+\psi(x-f-r).$$
It follows from the above and from the definition of $R(x)$ that $r\in R(x)$ if and only if $\max_{i\in I}a_i r+\max_{i\in I}a_i(x-f-r)\leq \psi(x-f)$. That is, $r\in R(x)$ if and only if $a_ir+a_j(x-f-r)\leq \psi(x-f)$ for all $i,j\in I$.
\medskip

\noindent $ii)$ Let $x\in \R^n$. By $i)$, a vector $r\in\R^n$ belongs to $\rec(R(x))$ if and only if $a_ir-a_jr\leq 0$ for all $i,j\in I$. The latter condition is verified if and only if, for all $i,j\in I$, $r$ satisfies $a_ir-a_jr\leq 0$ and $a_jr-a_ir\leq 0$, that is, if and only if $r\in L_\psi$. This shows that $\lin(R(x))=\rec(R(x))=L_\psi$.\medskip

\noindent $iii)$ If $x-x'\in L_\psi$, then, for all $i\in I$, $a_i(x-x')=\alpha$ for some constant $\alpha$. For all $r\in \R^n$, it follows that $a_i (x-f-r)=a_i(x-x')+a_i(x'-f-r)=\alpha+a_i(x'-f-r)$, $i\in I$. Choosing $r=0$, we also have that $a_i(x-f)=\alpha+a_i(x'-f)$ for all $i \in I$ and so $\psi(x-f) = \psi(x'-f)$. Therefore, for all $i,j\in I$,
$a_ir+a_j(x-f-r)\leq \psi(x-f)$ if and only if $a_ir+a_j(x'-f-r)\leq \psi(x'-f)$. By $i)$, this implies that  $r\in R(x)$ if and only if $r\in R(x')$.
\medskip

\noindent $iv)$ Let $B := B_\psi$ and let $L:=\rec(B\cap \conv(S))$. By Theorem~\ref{Th:minimal1}, $L$ is a rational subspace contained in $\lin(B)$. Since $\lin(B)=\{r\in \R^n\st a_ir=0 \mbox{ for all } i\in I \}$, we have that $L\subseteq \lin(B)\subseteq L_\psi$.

Let $P$ be the projection of $B$ onto the orthogonal complement $L^\bot$ of $L$, and let $Q$ be the projection of $\conv(S)$ onto $L^\bot$. Since $L=\rec(B\cap\conv(S))$, it follows that $P\cap Q$ is a polytope.

Given two elements $x,x'\in S$ whose orthogonal projections onto $L^\bot$ coincide, it follows that $x-x'\in L\subseteq L_\psi$, and therefore by $iii)$ $R(x)=R(x')$. It follows that the number of sets $R(x)$, $x\in S\cap B$, is at most the cardinality of the orthogonal projection $\tilde S$  of $S\cap B$ onto $L^\bot$.

Let $\Lambda$ be the orthogonal projection of $\Z^n$ onto $L^\bot$. Since $L$ is a rational space, it follows that $\Lambda$ is a lattice (see for example~\cite{barv}). In particular, there exists $\varepsilon>0$ such that $\|y-z\|\geq \varepsilon$ for all $y,z\in \Lambda$. Since $P\cap Q$ is a polytope, it follows that $P\cap Q\cap \Lambda$ is finite. Since $\tilde S\subseteq P\cap Q\cap \Lambda$, it follows that $\tilde S$ is a finite set.

We conclude that the family of polyhedra $R(x)$, $x\in S\cap B$, has a finite number of elements, thus $R_\psi=\bigcup_{x\in S \cap B}R(x)$ is the union of a finite number of polyhedra.
\medskip

\noindent $v)$ By $ii)$, for every $r\in R_\psi$, $\{r\}+L_\psi$ is contained in $R_\psi$. It was proved in~\cite{ccz} that $R_\psi$ contains the origin in its interior. That is, there exists $\varepsilon>0$ such that $r\in R_\psi$ for all $r$ such that $\|r\|< \varepsilon$. It then follows that  $\{r\st\|r\|<\varepsilon\}+L_\psi\subseteq R_\psi$, hence $L_\psi$ is contained in the interior of $R_\psi$.
\end{proof}

We observe that Proposition~\ref{prop:region} i) implies that, for every $x\in\R^n$, the region $R(x)$ is the intersection of a polyhedral cone $C$ with the translation of $-C$ by $x-f$. More formally, let $k\in I$ such that $\psi(x-f)=a_k(x-f)$, and let $C:=\{r\in\R^n\st (a_i-a_k)r\leq 0,\, i\in I\}$. We claim that $R(x)=C\cap ((x-f)-C)$, where $(x-f)-C:=\{r\st x-f-r\in C\}$. Observe that $(x-f)-C=\{r\st (a_i-a_k)(x-f-r)\leq 0,\, i\in I\}$. To show  $R(x)\subseteq C\cap ((x-f)-C)$, note that, for all $i\in I$, the inequalities $(a_i-a_k)r\leq 0$ and $(a_i-a_k)(x-f-r)\leq 0$ are equivalent to the inequalities $a_ir+a_k(x-f-r)\leq \psi(x-f)$ and $a_kr+a_i(x-f-r)\leq \psi(x-f)$, respectively. To show $ C\cap ((x-f)-C)  \subseteq R(x)$, note that, for all $i,j\in I$, the inequality $a_ir-a_j(x-f-r)\le \psi(x-f)$ is the sum of the inequalities $(a_i-a_k)r\leq 0$ and  $(a_j-a_k)(x-f-r)\leq 0$.

\begin{prop}\label{prop:pi_continuous}
Let $\psi$ be a minimal valid function for \eqref{mod:cont}. Every minimal lifting for $\psi$ is a continuous function.
\end{prop}
\begin{proof}
Let $\pi$ be a minimal lifting of $\psi$. It means that the function $(\psi,\pi)$ is a minimal valid function for~\eqref{mod:int}.
It is known that, if $(\psi,\pi)$ is a minimal valid function for~\eqref{mod:int}, then $\pi$ is subadditive. The latter fact was shown by Johnson~\cite{johnson} for the case $S=\Z^n$, but the proof for the general case is identical (see also~\cite{survey} for a proof). To prove that $\pi$ is continuous, we need to show that, given $\tilde r\in \R^n$, for every $\delta>0$ there exists $\varepsilon>0$ such that $| \pi(r)-\pi(\tilde r)|<\delta$ for all $r\in\R^n$ satisfying $\|r-\tilde r\|<\varepsilon$. Since $\psi$ is a continuous function, $\psi(0)=0$ and $\psi$ coincides with $\pi$ in some open ball containing the origin, it follows that, for every $\delta>0$, there exists $\varepsilon>0$ such that, for all $r\in\R^n$ satisfying $\|r-\tilde r\|<\varepsilon$, we have $|\psi(r-\tilde r)|<\delta$, $\pi(r-\tilde r)=\psi(r-\tilde r)$, and $\pi(\tilde r-r)=\psi(\tilde r-r)$.
Hence, for all $r\in\R^n$ satisfying $\|r-\tilde r\|<\varepsilon$, $$|\pi(r)-\pi(\tilde r)|\leq \max\{\pi(\tilde r-r),\pi(r-\tilde r)\}=\max\{\psi(r-\tilde r),\psi(\tilde r- r)\}<\delta,$$
where the first inequality follows from the subadditivity of $\pi$, since $\pi(r)\leq \pi(\tilde r)+\pi(r-\tilde r)$ and $\pi(\tilde r)\leq \pi(r)+\pi(\tilde r-r)$.
\end{proof}

\section{Minimum lifting coefficient of a single variable}

Given $r^*\in\R^n$, we consider the set of solutions to

\begin{eqnarray}
&x=&f+\sum_{r\in \R^n} r s_r+ r^* y_{r^*}\nonumber\\
&&x\in S\nonumber\\
&&s\geq 0\label{eq:oneray}\\
&&y_{r^*}\geq 0,\, y_{r^*}\in\Z\nonumber\\
&&s \mbox{ has finite support.}\nonumber
\end{eqnarray}

Given a minimal valid function $\psi$ for~\eqref{mod:cont} and
scalar $\lambda$, we say that the inequality
$\sum_{r\in\R^n}\psi(r)s_r+\lambda y_{r^*}\geq 1$ is valid
for~\eqref{eq:oneray} if it holds for every  $(x,s,y_{r^*})$
satisfying~\eqref{eq:oneray}. We denote by $\pi^*(r^*)$ the minimum
value of $\lambda$ for which $\sum_{r\in\R^n}\psi(r)s_r+\lambda
y_{r^*}\geq 1$ is valid for~\eqref{eq:oneray}.

We observe that,  for {\em any}
lifting $\pi$ of $\psi$, we have
$$\pi^*(r^*)\leq \pi(r^*).$$
Indeed, $\sum_{r\in\R^n}\psi(r) s_r+\pi(r^*)  y_{r^*}\geq 1$ is valid for~\eqref{eq:oneray}, since, for any $(\bar s,\bar y_{r^*})$ satisfying~\eqref{eq:oneray}, the vector $(\bar s,\bar y)$,  defined by $\bar y_r=0$ for all $r\in \R^n\setminus\{r^*\}$, satisfies~\eqref{mod:int}.

It  is easy to show (see~\cite{ccz}) that, if $\psi$ is a minimal valid function for~\eqref{mod:cont} and $\pi$ is a minimal
lifting of $\psi$, then $\pi \leq \psi$.

So, for every minimal valid function $\psi$ for~\eqref{mod:cont} and every minimal lifting $\pi$ of $\psi$, we have the following relation, $$\pi^*(r)\leq \pi(r)\leq \psi(r) \quad\textrm{for all }r
\in \R^n.$$ In general $\pi^*$ is not a lifting of $\psi$, but if it
is, then the above relation implies that it is the unique minimal
lifting of $\psi$.

\begin{lemma}\label{rmk:psi} For any $r\in \R^n$ such that $\pi^*(r)=\psi(r)$, we have $\pi(r)=\psi(r)$ for every minimal lifting $\pi$ of $\psi$. Conversely, if $\pi^*(r^*)<\psi(r^*)$ for some $r^*\in \R^n$, then there exists some minimal lifting $\pi$ of $\psi$ such that $\pi(r^*)=\pi^*(r^*)<\psi(r^*)$.\end{lemma}
\begin{proof}
The first part follows from $\pi^*\leq \pi\leq \psi$. For the second part,  given $r^*\in \R^n$ such that $\pi^*(r^*)<\psi(r^*)$, we can define a function $\pi'\,:\, \R^n\rightarrow \R$ by $\pi'(r^*)=\pi^*(r^*)$ and $\pi'(r)=\psi(r)$ for all $r\in\R^n$, $r\neq r^*$. Since $\psi$ is valid for~\eqref{mod:cont}, it follows by the definition of $\pi^*(r^*)$ that $\pi'$ is a lifting of $\psi$. As observed in the introduction, there exists a minimal lifting $\pi$ such that $\pi\leq \pi'$. Since $\pi^*\leq \pi$ and $\pi^*(r^*)=\pi'(r^*)$, it follows that $\pi(r^*)=\pi^*(r^*)$.\end{proof}

Next we present a geometric characterization, introduced in~\cite{ccz}, of the function $\pi^*$.
Let $r^*\in \R^n$. Given a maximal
$S$-free convex set $B=\{x\in\R^n\st a_i(x-f)\leq 1,\,\, i\in I\}$, for any $\lambda\in\R$, we define the set $B(\lambda,r^*)\subset\R^{n+1}$ as
follows
\begin{equation}
\label{eq:B-lambda} B(\lambda, r^*)=\{{x \choose x_{n+1}}\in\R^{n+1}\st
a_i(x-f)+(\lambda-a_i r^*)x_{n+1}\leq 1,\; i\in I\}.\end{equation}

Observe that $B(\lambda,r^*)\cap (\R^n\times\{0\})=B\times\{0\}$. For $\lambda=0$, the vector ${r\choose 1}$ belongs to the lineality space of $B(\lambda,r^*)$, thus $B(0,r^*)=(B\times\{0\})+L$ where $L$ is the line of $\R^{n+1}$ in the direction ${r\choose 1}$. For $\lambda\neq 0$, the point ${f\choose 0}+\lambda^{-1}{r\choose 1}$ satisfies at equality all the inequalities $a_i(x-f)+(\lambda-a_i r^*)x_{n+1}\leq 1$, $i\in I$, therefore  $B(\lambda,r^*)$ has a unique minimal face $F$ (i.e., it is the translation of a polyhedral cone), and ${f\choose 0}+\lambda^{-1}{r\choose 1}\in F$. The following theorem, has been proved in~\cite{ccz}.

\begin{theorem}\label{thm:B-lambda} Let $r^*\in \R^n$. Given a maximal
$S$-free convex set $B$ with $f$ in its interior, let $\psi=\psi_B$. Given $\lambda\in\R$, the inequality
$\sum_{r\in\R^n}\psi(r)s_r+\lambda y_{r^*}\geq 1$
 is valid for~\eqref{eq:oneray} if and only if $B(\lambda, r^*)$ is $S\times\Z_+$-free.
\end{theorem}
In particular, by the above theorem, $\pi^*(r^*)$ is the smallest possible $\lambda$ such that $B(\lambda, r^*)$ is $S\times\Z_+$-free. Observe that, as $\lambda$ gets smaller, the set $B(\lambda,r^*)\cap(\R^n\times\R_+)$ becomes larger, so for $\lambda$ sufficiently small $B(\lambda,r^*)$ will not be $S\times\Z_+$-free. Intuitively, there will be a ``blocking point'' in $S\times \Z_+$ that will not belong to $B(\lambda,r^*)$ as long as $\lambda>\pi^*(r^*)$, it will be on the boundary of $B(\lambda,r^*)$ for  $\lambda=\pi^*(r^*)$, and it will be in the interior for $\lambda<\pi^*(r^*)$. This fact is proven in Theorem~\ref{lemma:block-point} below, and it will play a central role in the proofs of  Theorem~\ref{thm:main} and~\ref{thm:main2}.
The next lemma will be needed in the proof of Theorem~\ref{lemma:block-point} and of Theorem~\ref{thm:main2}.

\begin{lemma}\label{lemma:B-lambda-max}  Let $r^*\in \R^n$, and let $B$ be a maximal $S$-free
convex set with $f$ in its interior. For every $\lambda$ such that $B(\lambda, r^*)$ is
$S\times\Z_+$-free, $B(\lambda, r^*)$ is maximal $S\times\Z_+$-free.
\end{lemma}
\begin{proof}
Since $B$ is a maximal $S$-free convex set, then by Theorem~\ref{Th:minimal1} each facet of $B$ contains a point $\bar x$ of $S$ in its relative interior. Therefore the corresponding facet of $B(\lambda, r^*)$ contains the point ${\bar x\choose 0}$ in its relative interior. If $B(\lambda, r^*)$ is $S\times\Z_+$-free, by Theorem~\ref{Th:minimal1} it is a maximal $S\times\Z_+$-free convex set.
\end{proof}

\begin{theorem}\label{lemma:block-point} Consider $r^*\in \R^n$ and $\lambda \in \R$.
We have $\lambda = \pi^*(r^*)$  if and only if $B(\lambda,r^*)$ is
$S\times\Z_+$-free and contains a point ${\bar x\choose\bar x_{n+1}}\in S\times \Z_+$ such that $\bar x_{n+1}>0$.
\end{theorem}
\begin{proof}
We first prove that, if $B(\lambda,r^*)$ is
$S\times\Z_+$-free and contains a point ${\bar x\choose\bar x_{n+1}}\in S\times \Z_+$ such that $\bar x_{n+1}>0$, then $\lambda = \pi^*(r^*)$.

Since $B(\lambda,r^*)$ is $S\times\Z_+$-free, Theorem~\ref{thm:B-lambda} and the definition of $\pi^*$ imply that
$\lambda \geq \pi^*(r^*)$. We claim that ${\bar x\choose\bar x_{n+1}}$ is in the interior of $B(\lambda -\epsilon,r^*)$ for every $\epsilon > 0$. Indeed, for all $i \in I$,
$a_i(\bar x-f)+(\lambda-\epsilon -a_i r^*)\bar x_{n+1} = a_i(\bar x-f)+(\lambda -a_i r^*)\bar x_{n+1}
-\epsilon \bar x_{n+1} < 1$, where the inequality holds because ${\bar x\choose\bar x_{n+1}} \in B(\lambda,r^*)$ implies that $a_i(x-f)+(\lambda -a_i r^*)x_{n+1}\leq 1$ and $\bar x_{n+1}>0$ implies $\epsilon \bar x_{n+1} > 0$. This proves our claim. Since ${\bar x\choose\bar x_{n+1}}$ is in the interior of $B(\lambda -\epsilon,r^*)$ for every $\epsilon > 0$, Theorem~\ref{thm:B-lambda} and the definition of $\pi^*$ imply that $\lambda - \epsilon < \pi^*(r^*)$. Hence $\lambda = \pi^*(r^*)$.
\smallskip

We now prove the converse. Let $\lambda^*=\pi^*(r^*)$.
The definition of $\pi^*$  and Theorem~\ref{thm:B-lambda} imply that
$B(\lambda^*,r^*)$ is $S\times\Z_+$-free. It only remains to show that
$B(\lambda^*,r^*)$ contains a point ${\bar x\choose\bar x_{n+1}}\in S\times \Z_+$ such that $\bar x_{n+1}>0$.

Note that, for all $\lambda\in \R$, $B(\lambda,r^*)\cap (\R^n\times\{0\})=B\times\{0\}$.

We consider two possible cases: either there exists ${\bar r\choose\bar r_{n+1}}\in
\rec(B(\lambda^*, r^*)\cap\conv(S\times\Z_+))$ such that $\bar r_{n+1}>0$, or $\rec(B(\lambda^*, r^*)\cap\conv(S\times\Z_+))=\rec(B\cap\conv(S))\times\{0\}$.
\bigskip

\noindent{\sl Case 1. } There exists ${\bar r\choose\bar r_{n+1}}\in
\rec(B(\lambda^*, r^*)\cap\conv(S\times\Z_+))$ such that $\bar r_{n+1}>0$.
\medskip

Note that $\rec(\conv(S\times\Z_+))=\rec(\conv(S))\times\R_+$, thus $\bar r\in\rec(\conv(S))$. By Lemma~\ref{lemma:B-lambda-max}, $B(\lambda^*, r^*)$ is a maximal
$S\times\Z_+$-free convex set, thus by Theorem~\ref{Th:minimal1}, $\rec(B(\lambda^*, r^*)\cap\conv(S\times\Z_+))$ is rational. Hence, we can choose ${\bar r\choose\bar r_{n+1}}$ integral. Since $B$ is a maximal $S$-free convex set, there exists $\tilde x\in S\cap B$.  Let ${\bar x\choose \bar x_{n+1}}={\tilde  x\choose 0}+{\bar r\choose \bar r_{n+1}}$. Since $\bar r\in\rec(\conv(S))$ and $\tilde x\in S$, it follows that $\bar x\in\conv(S)$. Since $\bar x\in \Z^n$, we conclude that $\bar x\in S$. Furthermore, $\bar x_{n+1}=\bar r_{n+1}$, thus $\bar x_{n+1}\in \Z$ and $\bar x_{n+1}>0$.
Finally, since ${\tilde  x\choose 0}\in B(\lambda^*,r^*)$ and ${\bar r\choose\bar r_{n+1}}\in
\rec(B(\lambda^*, r^*))$, we conclude that ${\bar x\choose \bar x_{n+1}}\in B(\pi^*(r^*),r^*)$.
\bigskip

\noindent{\sl Case 2. } $\rec(B(\lambda^*, r^*)\cap\conv(S\times\Z_+))=\rec(B\cap\conv(S))\times\{0\}$.
\bigskip

\noindent {\bf Claim. } $\exists$ $\bar \varepsilon>0$
such that
$\rec(B(\lambda^*-\bar\varepsilon, r^*)\cap\conv(S\times\Z_+))=\rec(B\cap\conv(S))\times\{0\}$.

\medskip
Since $\conv(S)$ is a polyhedron, $\conv(S)=\{x\in\R^n\st Cx\leq
d\}$ for some matrix $(C,d)$. By assumption, there is no vector ${r\choose 1}$ in
$\rec(B(\lambda^*, r^*)\cap\conv(S\times\Z_+))$. Thus the system
$$\begin{array}{rcll}
a_i r+(\lambda^*-a_ir^*)&\leq&0,&i\in I\\
Cr&\leq &0&
\end{array}$$
is infeasible. By Farkas' Lemma, there exist scalars $\mu_i\geq 0$,
$i\in I$ and a nonnegative vector $\gamma$ such that
\begin{eqnarray*}
\sum_{i\in I}\mu_ia_i+\gamma C&=&0\\
\lambda^*(\sum_{i\in I}\mu_i)-(\sum_{i\in I}\mu_i a_i)r^*&>&0.
\end{eqnarray*}
This implies that there exists some $\bar \varepsilon > 0$ small enough such that
\begin{eqnarray*}
\sum_{i\in I}\mu_ia_i+\gamma C&=&0\\
(\lambda^*-\bar\varepsilon)(\sum_{i\in I}\mu_i)-(\sum_{i\in I}\mu_i a_i)r^*&>&0,
\end{eqnarray*}
thus the system
$$\begin{array}{rcll}
a_i r+(\lambda^*-\bar\varepsilon-a_ir^*)&\leq&0,&i\in I\\
Cr&\leq &0&
\end{array}$$
is infeasible. This implies that
$\rec(B(\lambda^*-\bar\varepsilon, r^*)\cap\conv(S\times\Z_+))=\rec(B\cap\conv(S))\times\{0\}$.
This concludes the proof of the claim.
\medskip

By the previous claim, there exists
$\bar\varepsilon$ such that
$\rec(B(\lambda^*-\bar\varepsilon, r^*)\cap\conv(S\times\Z_+))=\rec(B\cap\conv(S))\times\{0\}$. This implies that there exists a scalar $M$ such that $\bar x_{n+1}\leq M$ for every point ${\bar
x\choose\bar x_{n+1}}\in B(\lambda^*-\bar\varepsilon, r^*)\cap
(S\times\Z_+)$.

Remark~\ref{rmk:exisits-delta} and Lemma~\ref{lemma:B-lambda-max}
imply that there exists $\delta>0$ such that, for every ${\bar
x\choose\bar x_{n+1}}\in (S\times\Z_+)\setminus B(\lambda^*, r^*)$, there
exists $h\in I$ such that $a_h (\bar x-f)+(\lambda^*-a_hr^*)\bar
x_{n+1}\geq 1+\delta$. Choose $\varepsilon>0$ such that
$\varepsilon\leq\bar\varepsilon$ and $\varepsilon M\leq \delta$.

Since $\pi^*(r^*)=\lambda^*$,  Theorem~\ref{thm:B-lambda} implies that
$B(\lambda^*-\varepsilon, r^*)$ has a point ${\bar x\choose\bar
x_{n+1}}\in S\times \Z_+$ in its interior. Thus $a_i(\bar
x-f)+(\lambda^*-\varepsilon-a_ir^*)\bar x_{n+1}<1$, $i\in I$.

We show that ${\bar x\choose\bar x_{n+1}}$ is also in
$B(\lambda^*,r^*)$. Suppose not. Then, by our choice of $\delta$, there
exists $h\in I$ such that $a_h (\bar x-f)+(\lambda^*-a_hr^*)\bar
x_{n+1}\geq 1+\delta$.

It follows from the  definition of $B(\lambda,r^*)$ given in~\eqref{eq:B-lambda} that $B(\lambda^*-\varepsilon, r^*)\cap
(\R^n\times\R_+)\subseteq B(\lambda^*-\bar\varepsilon, r^*)\cap
(\R^n\times\R_+)$ because $\varepsilon \leq \bar\varepsilon$; therefore $\bar x_{n+1}\leq M$. Hence
$$1+\delta\leq a_h (\bar x-f)+(\lambda^*-a_hr^*)\bar x_{n+1}\leq a_h (\bar x-f)+(\lambda^*-\varepsilon-a_hr^*)\bar x_{n+1}+\varepsilon M<1+\varepsilon M\leq 1+\delta,$$
a contradiction.

Hence ${\bar x\choose\bar x_{n+1}}$ is in $B(\lambda^*,r^*)$. Since
$B$ is $S$-free and $B(\lambda^*-\varepsilon, r^*)\cap(\R^n\times\{0\})=B\times\{0\}$,
it follows that $B(\lambda^*-\varepsilon, r^*)$ does not contain any point of $S\times \{0\}$ in its
interior. Thus  $\bar x_{n+1}>0$.
\end{proof}

\section{Proof of Theorem~\ref{thm:main}}
\label{sec:Thm-lifting-region}

The proof of Theorem~\ref{thm:main} will follow easily from the next proposition.
 \begin{prop}\label{prop:equivalence} Let $\psi$ be a  minimal valid function for~\eqref{mod:cont}, and let $r^*\in\R^n$. The following are equivalent.
\begin{itemize}
\item[i)] $\pi^*(r^*) = \psi(r^*)$.
\item[ii)] There exists $x\in S \cap B_\psi$ such that $\psi(r^*)+\psi(x-f-r^*)=\psi(x-f)=1$.
\end{itemize}
\end{prop}
\begin{proof}
We first prove ii) implies i). The vector $(s,y_{r^*})$ defined by $y_{r^*}=1$, $s_{x-f-r^*}=1$, $s_r=0$ for all $r \not=x-f-r^*$, is a solution of \eqref{eq:oneray}. Since $\sum_{r\in \R^n}\psi(r)s_r+\pi^*(r^*)y_{r^*}\geq 1$ is valid for \eqref{eq:oneray}, it follows that
$$1\leq \pi^*(r^*)+\psi(x-f-r^*)\leq \psi(r^*)+\psi(x-f-r^*)=1,$$
thus $\pi^*(r^*) = \psi(r^*)$.

\bigskip
We now prove i) implies ii). By Theorem~\ref{lemma:block-point},
$B(\pi^*(r^*),r^*)$ contains a point ${x\choose x_{n+1}}\in S\times \Z_+$
such that $ x_{n+1}>0$. Since $\pi^*(r^*) = \psi(r^*)=\max_{i\in I}a_ir^*$,
 the coefficients of $x_{n+1}$ in the inequalities \eqref{eq:B-lambda} defining $B(\pi(r^*),r^*)$ are all nonnegative.
 This shows that  $B(\pi^*(r^*),r^*)$ contains the point ${x\choose 1}$ and therefore
$a_i (x-f) + \psi (r^*) - a_i r^* \leq 1$ for all $i \in I$. This implies that
$\max_{i \in I} a_i (x-f-r^*) + \psi (r^*) \leq 1$. Hence we have that $\psi(x-f-r^*)+\psi(r^*) \leq 1$. Conversely, we also have $1 \leq \psi (x-f) \leq \psi(x-f-r^*)+\psi(r^*)$, where the first inequality follows from the fact that  $x \in S$ and that, by Theorem~\ref{Th:minimal2}, $B_\psi$ is $S$-free, while the second inequality follows from the fact that $\psi$ is subadditive.
We conclude that $\psi(r^*)+\psi(x-f-r^*)=\psi(x-f)=1$.
\end{proof}

\begin{proof}[Proof of Theorem~\ref{thm:main}]
By Lemma~\ref{rmk:psi},  $R_\psi=\{r\in \R^n: \pi^*(r)=\psi(r)\}$. By Proposition~\ref{prop:equivalence}, given $r\in\R^n$,  $\pi^*(r)=\psi(r)$ if and only if $r\in R(x)$ for some $x\in S\cap B_\psi$. Thus $R_\psi=\cup_{x\in S\cap B_\psi}R(x)$.
\end{proof}
\section{Proof of Theorem~\ref{thm:main2}}
\label{sec:unique-lifting}

Let $\psi$ be a minimal valid function for \eqref{mod:cont}
and let $B := B_\psi=\{x\in\R^n\st a_i(x-f)\leq 1,\, i\in I\}$. In this section we assume $S=\Z^n$. Under this assumption, Theorem~\ref{Th:minimal1} shows that:
\begin{itemize}
\item  B is a maximal lattice-free convex set.
\item  $\rec(B)=\lin(B)=\{r\in \R^n\st a_ir=0 \mbox{ for all } i\in I \}$ and $\rec(B)$ is a rational subspace.
\item $\psi(r)=\max_{i\in I}a_ir\ge 0$ for every $r\in \R^n$.
\end{itemize}
\iffalse
By Theorem~\ref{Th:minimal1}, $B$ is a maximal lattice-free convex set containing $f$ in its interior.
Following \eqref{eq:unique-rep-B}, it can be uniquely written
as $B_\psi=\{x\in\R^n\st a_i(x-f)\leq 1,\, i\in I\}$.
Since $S=\Z^n$, by Theorem~\ref{Th:minimal1}, $\rec(B)=\lin(B)$. This implies that $\psi(r) = \max_{i \in I} a_ir \geq 0$ for all $r \in \R^n$.
\fi
Let $L_\psi:=\{r\in \R^n\st a_ir=a_jr \mbox{ for all } i,j\in I \}$.
To prove Theorem~\ref{thm:main2}, we need the following three lemmas.

\begin{lemma} \label{lemma:closed-set} $L_\psi=\lin(B)$. Furthermore, $R_\psi +\Z^n$ is a closed set.
\end{lemma}
\begin{proof}
Let $L :=\lin (B)$. It follows from Theorem~\ref{Th:minimal1} that $L$ is a rational linear subspace and $\rec(B)=L$. Given $r\in L_\psi$, it follows that $a_i r=a_j r$ for all $i,j\in I$, which implies that $a_ir= \alpha$ for all $i\in I$ for some constant $\alpha$ that only depends on $r$. If $\alpha\leq 0$ then $r\in \rec(B)=L$, while if $\alpha\geq 0$ then $-r\in \rec(B)=L$, thus $r\in L$. The converse inclusion $L \subseteq L_\psi$ is trivial.

Next we show that $R_\psi +\Z^n$ is a closed set. It follows from Proposition~\ref{prop:region} that $R_\psi$ is the union of a finite number of polyhedra, $R_1,\ldots,R_k$, and that $\rec(R_i)=L$ for $i=1,\ldots,k$. Let $\tilde R_i$ be the orthogonal projection of $R_i$ onto $L^\bot$. It follows that $\tilde R_i$ is a polytope, so in particular $\tilde R_i$ is compact. Furthermore, $R_i=\tilde R_i+L$. Let $\Lambda$ be the orthogonal projection of $\Z^n$ onto $L^\bot$. Since $L$ is rational, $\Lambda$ is a lattice, so in particular it is closed. Note that $R_i+\Z^n=(\tilde R_i+\Lambda)+L$. Since the Minkowski sum of a compact set and a closed set is closed (see e.g.~\cite{inf-dim-an}  Lemma 5.3 (4)), it follows that $\tilde R_i+\Lambda$ is closed. Since the Minkowski sum of a closed set and a linear subspace is closed, it follows that $(\tilde R_i+\Lambda)+L$ is closed.
\end{proof}

\begin{lemma}\label{claim:boundary}
Let $H:=\R^n\sm (R_\psi + \Z^n)$. If $H\neq \emptyset$, then there exists $\bar r \in R_\psi$ such that $\bar r$ belongs to the closure of $H$.
Any such vector satisfies  $\psi(\bar r)>0$.
\end{lemma}
\begin{proof} By Lemma~\ref{lemma:closed-set}, $R_\psi + \Z^n$ is closed. It follows that, if $R_\psi + \Z^n\neq \R^n$, then there exists a point $\tilde r\in R_\psi + \Z^n$ such that  $\tilde r$ belongs to the closure of $H$. Since $\tilde r\in R_\psi + \Z^n$, there exists $\bar r\in R_\psi$ and $w\in\Z^n$ such that $\tilde r=\bar r+w$. We show that the point $\bar r$ belongs to the closure of $H$. If not, then there exists an open ball $C$ centered at $\bar r$ contained in the interior of $R_\psi$, thus $C+w$ is an open ball centered at $\tilde r$ contained in the interior of $R_\psi+\Z^n$, contradicting the fact that $\tilde r$ is in the closure of $H$.

Finally, given $\bar r$ on the boundary of $R_\psi$, Proposition~\ref{prop:region} v) and the fact that $L_\psi = \lin(B)$ imply that $\bar r\notin \lin(B)$, thus $\psi(\bar r)>0$.
\end{proof}

\begin{lemma}\label{lemma:bounded-rec}
Let $r^*\in\R^n$. If $\pi^*(r^*)>0$, then $\rec(B(\pi^*(r^*), r^*)\cap(\R^n\times \R_+))=\rec(B)\times\{0\}$.
\end{lemma}
\begin{proof}
Let $\lambda^*=\pi^*(r^*)>0$.  Let ${r\choose r_{n+1}}\in \R^n\times \R_+$ such that ${r\choose r_{n+1}}\in \rec(B(\lambda^*,r^*))$. If $r_{n+1}=0$, it follows that $r\in \rec(B)$, thus ${r\choose r_{n+1}}\in \rec(B)\times\{0\}$. So suppose by contradiction that $r_{n+1}>0$. We may assume without loss of generality that $r_{n+1}=1$. Since $B(\lambda^*,r^*)$ is a maximal $\Z^n\times \Z_+$-free convex set, it follows from Theorem~\ref{Th:minimal1} that $a_i r+(\lambda^*-a_i r^*)=0$ for all $i\in I$. Since $\lambda^*>0$, the point $\hat r:=\frac{r^*-r}{\lambda^*}$ satisfies $a_i\hat r=1$ for all $i\in I$. It follows that $\hat r\notin\rec(B)$ while $-\hat r\in\rec(B)$, contradicting the fact that $\rec(B)=\lin(B)$.
\end{proof}

\begin{proof}[Proof of Theorem~\ref{thm:main2}]

As already mentioned in the introduction, it is shown in~\cite{ccz} that if $R_\psi+\Z^n=\R^n$ then $\pi^*$ is the unique minimal valid lifting of $\psi$.\medskip

We show the converse, that is, we show that if $R_\psi +\Z^n$ {\em does not} cover all of $\R^n$, then there exist at least two distinct minimal liftings of $\psi$.

We first observe that there exist two distinct minimal liftings of $\psi$ if and only if $\pi^*$ is not a  lifting for $\psi$. Indeed, if $\pi^*$ is a lifting for $\psi$, then $\pi^*$ is the unique minimal lifting for $\psi$, since any other lifting $\pi$ satisfies $\pi^*\leq \pi$. Conversely, suppose that  $\pi^*$ is not a lifting for $\psi$. Given any minimal lifting $\pi$ of $\psi$, since $\pi^*$ is not a valid lifting it follows that there exists $r'\in\R^n$ such that $\pi^*(r')<\pi(r')$. By Lemma~\ref{rmk:psi}, there exists a minimal lifting $\pi'$ such that $\pi'(r')=\pi^*(r')$. Thus $\pi$ and $\pi'$ are distinct minimal liftings of $\psi$.
\bigskip

Thus, we only need to show that, if $R_\psi +\Z^n\neq \R^n$, then $\pi^*$ is not a lifting. Suppose by contradiction that $R_\psi +\Z^n\neq \R^n$ but $\pi^*$ is a lifting of $\psi$. It follows that $\pi^*$ is a minimal lifting for $\psi$, thus by Proposition~\ref{prop:pi_continuous}, $\pi^*$ is a continuous function.

Let $H = \R^n \setminus (R_\psi + \Z^n)$. We will show the following.

\begin{claim}\label{claim:block-point}  There exists $p\in H$ such that $B(\pi^*(p),p)$ contains a point ${\bar x\choose 1}$ with $\bar x\in\Z^n$.
\end{claim}

Before proving Claim~\ref{claim:block-point}, we use it to conclude the proof of Theorem~\ref{thm:main2}. By Claim~\ref{claim:block-point}, there exists $p\in\R^n\setminus(R_\psi+\Z^n)$ such that $B(\pi^*(p),p)$ contains a point ${\bar x\choose 1}$ with $\bar x\in\Z^n$. It follows that $a_i(\bar x - f) +
(\pi^*(p) - a_ip) \leq 1$ for all $i\in I$. Moreover, by Theorem~\ref{thm:B-lambda}, $B(\pi^*(p), p)$ is $S\times \Z_+$-free,  thus there exists
$h\in I$ such that $a_h(\bar x - f) + (\pi^*(p) - a_hp) = 1$. This
shows that $a_i(\bar x - f - p) \leq 1 - \pi^*(p)$ for all $i\in I$
and $a_h(\bar x - f - p) = 1 - \pi^*(p)$. Since $\psi(\bar x - f - p)= \max_{i\in I}a_i(\bar x - f - p) =a_h(\bar x - f - p)$, we have that
$$
\psi(\bar x - f - p) = 1 - \pi^*(p).
$$
Since $\pi^*$ is a lifting,  $\pi^*(p)+\pi^*(\bar x - f - p)\ge 1$ and since  $\pi^*(\bar x - f - p)\le \psi(\bar x - f - p)$, this shows that $\pi^*(\bar x - f - p)=\psi(\bar x - f - p)$. It follows that
$\bar x - f - p  \in R_\psi$. In particular, $\bar x - f - p\in  R(\tilde x)$ for some $\tilde x \in \Z^n \cap B$. By definition of $ R(\tilde x)$, $\tilde x - f - (\bar x - f - p) \in  R(\tilde x)$, that is, $\tilde x -\bar x + p \in  R(\tilde x)$. This implies that $p \in R(\tilde x) + \Z^n$, a contradiction.
\bigskip

The remainder of the proof is devoted to showing Claim~\ref{claim:block-point}. By Lemma~\ref{claim:boundary}, the closure of $H$ contains a point $\bar r \in R_\psi$ and such a vector satisfies $\pi^*(\bar r)=\psi(\bar r)>0$. Since, by Proposition~\ref{prop:pi_continuous}, $\pi^*$ is continuous, there exists $\bar\varepsilon>0$ such that, for every $r\in\R^n$ satisfying $\|r-\bar r\|\leq \bar\varepsilon$, $\pi^*(r)>0$.

Let $\mu=\min\{\pi^*(r)\st \|r-\bar r\|\leq \bar\varepsilon\}$. Note that $\mu$ is well defined since $\pi^*$ is continuous and $\{r\in\R^n\st \|r-\bar r\|\leq \bar\varepsilon\}$ is compact. Furthermore, by the choice of $\bar\varepsilon$, $\mu>0$. Let $M=\mu^{-1}$.

\begin{claim} \label{claim:2}
For every $r\in \R^n$ satisfying  $\|r-\bar r\|<\bar\varepsilon$ and every ${x\choose x_{n+1}}\in B(\pi^*(r),r)$,  we have $x_{n+1}\leq M$.
\end{claim}

\noindent Let $r\in\R^n$ such that  $\|r-\bar r\|<\bar\varepsilon$. By Lemma~\ref{lemma:bounded-rec}, $\max\{x_{n+1}\st {x\choose x_{n+1}}\in B(\pi^*(r),r)\}$ is bounded. The above is a linear program, thus the set of its optimal solutions contains a minimal face of $B(\pi^*(r),r)$. Since the point ${f\choose 0}+\frac{1}{\pi^*(r)}{r\choose 1}$ satisfies all inequalities $a_i(x-f)+(\pi^*(r)-a_ir)x_{n+1}\leq 1$ at equality, it follows that $B(\pi^*(r),r)$ has a unique minimal face and that ${f\choose 0}+\frac{1}{\pi^*(r)}{r\choose 1}$ is in it. It follows that ${f\choose 0}+\frac{1}{\pi^*(r)}{r\choose 1}$ is an optimal solution, thus $\max\{x_{n+1}\st {x\choose x_{n+1}}\in B(\pi^*(r),r))\}=\frac{1}{\pi^*(r)}\leq M$.
This concludes the proof of Claim \ref{claim:2}.
\bigskip

By Remark~\ref{rmk:exisits-delta} and Lemma~\ref{lemma:B-lambda-max}, there exists $\delta>0$  such that, for every ${\bar x\choose \bar x_{n+1}}\in(\Z^n\times\Z_+)\sm B(\psi(\bar r),\bar r)$, $a_i(\bar x-f)+(\psi(\bar r)-a_ir)\bar x_{n+1}\geq 1+\delta$, for some $i\in I$.

Since $\pi^*$ is a continuous function, $\pi^*(\bar r)=\psi(\bar r)$, and $\bar r$ is in the closure of $H$, there exists $p\in H$ such that $\|p-\bar r\|<\bar\varepsilon$ and $$|(\psi(\bar r)-a_i\bar r)-(\pi^*(p)-a_ip)|<\frac\delta M \mbox{ for all } i\in I.$$

By Theorem~\ref{lemma:block-point} and Claim~\ref{claim:2}, $B(\pi^*(p),p)$ contains a point ${\bar  x\choose \bar x_{n+1}}\in\Z^n\times \Z_+$ such that $0<\bar x_{n+1}\leq M$.
We conclude by showing that ${\bar  x\choose 1}$ is in $B(\pi^*(p),p)$, thus proving Claim~\ref{claim:block-point}.
\medskip

First we show that ${\bar  x\choose 1}$ is in $B(\psi(\bar r),\bar r)$. Indeed, for all $i\in I$,
\begin{equation}\label{eq:final}
\begin{array}{rcl}
a_i(\bar x-f)+(\psi(\bar r)-a_i\bar r) & \leq & a_i(\bar x-f)+(\psi(\bar r)-a_i\bar r)\bar x_{n+1} \\
& < &  a_i(\bar x-f)+(\pi^*(p)-a_i p+\frac \delta M)\bar x_{n+1} \\
& \leq & 1+\delta,
\end{array}\end{equation}
where the first inequality follows from the facts that $\psi(\bar r)\geq a_i \bar r$ and $\bar x_{n+1}\geq 1$, while the last inequality follows from the facts that ${\bar  x\choose \bar x_{n+1}}\in B(\pi^*(p),p)$ and $\bar x_{n+1}\leq M$. By our choice of $\delta$, the strict inequality in \eqref{eq:final} shows that ${\bar  x\choose 1}$ is in $B(\psi(\bar r),\bar r)$. In particular, since $\psi(\bar r)\geq a_i\bar r$ for all $i\in I$, it follows that $a_i(\bar x-f)\leq 1$ for all $i\in I$.

We finally show that  $a_i(\bar x-f)+\pi^*(p)-a_ip\leq 1$ for all $i\in I$, implying that ${\bar  x\choose 1}$ is in $B(\pi^*(p),p)$. Indeed, if $\pi^*(p)-a_ip \leq 0$, then $a_i(\bar x-f)+\pi^*(p)-a_ip\leq 1+\pi^*(p)-a_ip\leq 1$, while if $\pi^*(p)-a_ip > 0$, then $a_i(\bar x-f)+\pi^*(p)-a_ip\leq a_i(\bar x-f)+(\pi^*(p)-a_ip)\bar x_{n+1}\leq 1$.
\end{proof}

\section*{Acknowledgements}

The authors would like to thank Marco Molinaro for helpful discussions about the results
presented in this paper.

\end{document}